\documentclass[11pt,a4paper,oneside]{article}

\usepackage{authblk}

\title{On the number of solutions to a random instance of the permuted kernel problem%
	\footnote{%
		C.~Sanna is a member of GNSAGA of INdAM and of CrypTO, the group of Cryptography and Number~Theory of the Politecnico di Torino.
        This work was partially supported by project SERICS (PE00000014) under the MUR National Recovery and Resilience Plan funded by the European Union -- NextGenerationEU.}
}

\author{Carlo Sanna}

\affil{%
	Department of Mathematical Sciences, Politecnico di Torino\\
	Corso Duca degli Abruzzi 24, 10129 Torino, Italy
}

\affil[]{\texttt{carlo.sanna@polito.it}}

\date{} % no date

\usepackage{amsmath}
\usepackage{amssymb}
\usepackage{amsthm}
\usepackage{bm}
\usepackage{bbm}
\usepackage{comment}
\usepackage{geometry}
\usepackage[colorlinks=true]{hyperref}
\usepackage{enumitem}
\setlist[enumerate]{label=(\roman*),labelindent=1em,itemsep=0.5em,topsep=0.5em}
\usepackage{placeins}
\usepackage{graphicx}

\newtheorem{theorem}{Theorem}[section]

\newtheorem{lemma}[theorem]{Lemma}

\theoremstyle{definition}

\newtheorem{problem}{Problem}[section]
\theoremstyle{remark}

\DeclareMathOperator{\rank}{rk}
\DeclareMathOperator{\lker}{left-ker}
\newcommand{\transpose}{\intercal}
\newcommand{\GenIPKP}{\mathsf{GenIPKP}}
\newcommand{\GenPKP}{\mathsf{GenPKP}}
\newcommand{\sol}{\textsf{sol}}
\newcommand{\eqdistr}{\stackrel{\textsf{d}}{=}}

\begin{document}

\maketitle

\begin{abstract}
	The \emph{Permuted Kernel Problem} (PKP) is a problem in linear algebra that was first introduced by Shamir in 1989.
    Roughly speaking, given an $\ell \times m$ matrix $\bm{A}$ and an $m \times 1$ vector $\bm{b}$ over a finite field of $q$ elements $\mathbb{F}_q$, the PKP asks to find an $m \times m$ permutation matrix $\bm{\pi}$ such that $\bm{\pi} \bm{b}$ belongs to the kernel of $\bm{A}$.
    In recent years, several post-quantum digital signature schemes whose security can be provably reduced to the hardness of solving random instances of the PKP have been proposed.
    In this regard, it is important to know the expected number of solutions to a random instance of the PKP in terms of the parameters $q,\ell,m$.
    Previous works have heuristically estimated the expected number of solutions to be $m! / q^\ell$.

    We provide, and rigorously prove, exact formulas for the expected number of solutions to a random instance of the PKP and the related \emph{Inhomogeneous Permuted Kernel Problem} (IPKP), considering two natural ways of generating random instances.
\end{abstract}

{
    \small
    \noindent
    \textbf{Keywords:} cryptography; digital signatures; NP-complete problem; permutations; permuted \mbox{kernel} problem; post-quantum cryptography.\\[2pt]
    \textbf{MSC2020:} 05A05, 05A16, 15A99, 11T71, 68Q25.
    % 05-XX Combinatorics
    %  05Axx Enumerative combinatorics
    %   05A05 Permutations, words, matrices
    %   05A16 Asymptotic enumeration
    % 15Axx Basic linear algebra
    %  15A99 None of the above, but in this section
    % 11-XX Number theory
    %  11Txx Finite fields and commutative rings
    %   11T71 Algebraic coding theory; cryptography (number-theoretic aspects)
    % 68-XX Computer science
    %  68Qxx Theory of computing
    %   68Q25 Analysis of algorithms and problem complexity
}

\section{Introduction}\label{sec:intro}

The \emph{Permuted Kernel Problem} (PKP) is a problem in linear algebra that was first introduced by Shamir in 1989~\cite{C:Shamir89}.
Roughly speaking, given an $\ell \times m$ matrix $\bm{A}$ and an $m \times 1$ vector $\bm{b}$ over a finite field of $q$ elements $\mathbb{F}_q$, the PKP asks to find an $m \times m$ permutation matrix $\bm{\pi}$ such that $\bm{\pi} \bm{b}$ belongs to the kernel of $\bm{A}$, that is, $\bm{A} \bm{\pi} \bm{b} = \bm{0}$.
Using as security assumption the computational hardness of solving random instances $(\bm{A}, \bm{b})$ of the PKP, Shamir devised an identification scheme that has a very efficient implementation on low-cost smart cards.

In recent years, for several reasons, the PKP has become a very attractive problem to build post-quantum cryptographic schemes.
First, the PKP is based on simple objects and operations, which can be implemented easily and efficiently.
Second, the hardness of the PKP, and of some natural variants of the PKP, has been intensively studied \cite{C:BCCG92,JC:Georgiades92,PKC:JauJou01,cryptoeprint:2019/412,ACNS:PaiTer21,C:ParCha93,9834867,10.1109/TIT.2023.3323068}, and despite many efforts the best known algorithms have exponential complexities.
Third, the PKP is known to be NP-complete in the strong sense~\cite{C:Shamir89}, and quantum computers are expected to have a limited advantage in solving NP-complete problems (essentially, no more than the advantage of Grover's search)~\cite{MR1471991}.

Consequently, several post-quantum digital signature schemes whose security can be provably reduced to the hardness of solving random instances of the PKP have been proposed, namely: \textsf{PKP-DSS} (2019) \cite{INDOCRYPT:BFKMPP19}, \textsf{SUSHSYFISH} (2020) \cite{EC:Beullens20}, Bidoux and Gaborit's (2023)~\cite{10.1007/978-3-031-33017-9_2}, and \textsf{PERK} (2024) \cite{Bettaieb2024}.
In particular, \textsf{PERK} has been submitted to the NIST additional call for the post-quantum cryptography standardization process~\cite{NIST_call}.

In all these schemes, the public key is a random instance of the PKP (or a variant thereof), the secret key is a solution to such an instance, and the signing and verification algorithms constitute a non-interactive zero-knowledge proof of knowledge of the solution.
Furthermore, the parameters $q, \ell, m$ are selected in order to ensure that the known algorithms to solve the PKP require on average $2^\lambda$ or more operations, where $\lambda$ is the desired security parameter.

In addition, for mainly two reasons, it is desiderable to choose $q, \ell, m$ so that the random instance of the PKP is likely to have exactly one solution (or at least a number of solutions that is bounded by a known constant).
First, because it is natural to have exactly one secret key corresponding to the public key.
Second, because the complexity of an algorithm searching a solution to the PKP is expected to be (approximately) inversely proportional to the number of solutions.

Shamir~\cite{C:Shamir89} stated that if $q^\ell \approx m!$ then a random instance of the PKP is likely to have a unique solution.
This claim (or natural generalizations of it) is repeated in most of the subsequent works~\cite{cryptoeprint:2023/589,PKC:JauJou01,cryptoeprint:2019/412,EPRINT:LAMPAT11,ACNS:PaiTer21,C:ParCha93,https://doi.org/10.1002/ett.4460080505}.
The reasoning behind it is the following.
Assuming that $\ell \leq m$ and that $\bm{A}$ has rank equal to $\ell$, the probability that a random $m \times 1$ vector $\bm{c}$ over $\mathbb{F}_q$, which is taken with uniform distribution and independently from $\bm{A}$, belongs to the kernel of $\bm{A}$ is equal to $q^{-\ell}$.
Hence, assuming that, for a uniformly distributed random $m \times m$ permutation matrix $\bm{\pi}$, the vector $\bm{\pi}\bm{b}$ behaves like the random vector $\bm{c}$, by the linearity of the expectation we get that the expected number of solutions to the PKP instance $(\bm{A}, \bm{b})$ is equal to $m! / q^\ell$.

Note that the previous reasoning is only a heuristic and not a rigorous mathematical proof.
The main issue is that, unlike $\bm{c}$ and $\bm{A}$, the random variables $\bm{\pi}\bm{b}$ and $\bm{A}$ are not independent.
In fact, the probability of the event $\bm{A}\bm{\pi}\bm{b} = \bm{0}$ depends on how the random instance $(\bm{A}, \bm{b})$ is generated.

Moreover, for real-world choices of the parameters $q, \ell, m$, the heuristic formula $m! / q^\ell$ cannot be empirically tested.
In fact, since the parameters are chosen to make finding the solutions difficult, one cannot efficiently count the solutions to test the formula empirically. (Note that, since PKP is NP-complete, its search version reduces to its decision version, which in turn reduces to its counting version. Therefore, there should be no shortcuts to count the solutions much more efficiently than by actually finding them.)

In light of the previous considerations, and due to the importance of building post-quantum cryptography on solid mathematical foundations, the purpose of this paper is to provide, and rigorously prove, exact formulas for the expected number of solution to random instances of the PKP.
More precisely, we consider both the PKP and the \emph{Inhomogeneous Permuted Kernel Problem} (IPKP), and for each of these problems we study two natural ways of generating random instances.

The structure of the paper is as follows.
In Section~\ref{sec:notation}, we state the main notation.
In Section~\ref{sec:IPKP}, we provide the definition of the IPKP and we introduce the algorithms to generate random instances of the IPKP and the PKP.
In Section~\ref{sec:number-of-solutions}, we state the main results of the paper, that is, the exact formulas for the expected number of solutions to the random instances of the IPKP and the PKP generated by the algorithms.
Moreover, we compare these exact formulas with the heuristic formula.
In Section~\ref{sec:proofs}, we give the proofs of the main results.
Finally, in Section~\ref{sec:remarks}, we state some concluding remarks and some possible questions for future research.

\section{Notation}\label{sec:notation}

Hereafter, let $q$ be a prime power, let $\mathbb{F}_q$ be a finite field of $q$ elements, and let $\mathbb{F}_q^*$ be the multiplicative group of $\mathbb{F}_q$.
For all integers $m,n,r \geq 0$, let $\mathbb{F}_q^{m \times n}$ be the vector space of $m \times n$ matrices over $\mathbb{F}_q$, let $\mathbb{F}_q^{m \times n, r}$ be the subset of $\mathbb{F}_q^{m \times n}$ containing the matrices of rank equal to $r$, and let $\mathbb{F}_q^{\star m \times n, r}$ be the subset of $\mathbb{F}_q^{m \times n, r}$ containing the matrices having $m$ pairwise distinct nonzero rows.
Let $\bm{0}$ and $\bm{I}$ be the zero matrix and the identity matrix, respectively, where the sizes will be always clear from the context.
For every matrix $\bm{A} \in \mathbb{F}_q^{m \times n}$, let $\rank(\bm{A})$ be the rank of $\bm{A}$, let $\bm{A}^\transpose$ be the transpose of $\bm{A}$, let $\lker(\bm{A}) := \{\bm{x} \in \mathbb{F}_q^{1 \times m} : \bm{x}\bm{A} = \bm{0}\}$ be the left kernel of $\bm{A}$, and for every positive integer $k$ let $\big(\!\lker(\bm{A})\big)^k$ the set of $k \times m$ matrices with each row in $\lker(\bm{A})$.

Let $\mathbb{S}_n$ be the symmetric group of $\{1, \dots, n\}$.
For each permutation $\pi \in \mathbb{S}_n$, let $\bm{\pi}$ be the corresponding $n \times n$ permutation matrix whose entry of the $i$th row and $j$th column is equal to the Kronecker symbol $\delta_{\pi(i),j}$.
For each positive integer $k \leq n$, let $\mathbb{S}_{n,k}$ be the subset of permutations in $\mathbb{S}_n$ that can be written as the product of exactly $k$ disjoint cycles.

Let $s \gets \mathcal{S}$ denote that $s$ is taken at random with uniform distribution from the finite set $\mathcal{S}$.
Let $y \gets \textsf{A}(x)$ denote running the (possibly probabilistic) algorithm $\textsf{A}$ on input $x$ and assigning the output to $y$.

Let $X \eqdistr Y$ mean that the random variables $X$ and $Y$ have the same distribution, let $\mathbb{P}[E]$ be the probability that the event $E$ happens, let $\mathbb{E}[X]$ be the expected value of the random variable $X$, and let $\mathbbm{1}[S]$ be equal to $1$, respectively $0$, if the statement $S$ is true, respectively false.

Let $\varphi(\cdot)$ be the Euler function, and let $\lfloor\cdot\rfloor$ be the floor function.

\section{Inhomogeneous permuted kernel problem}\label{sec:IPKP}

The \emph{Inhomogeneous Permuted Kernel Problem} (IPKP) is defined as follows (cf.\ \cite{10.1109/TIT.2023.3323068}).

\begin{problem}[Inhomogeneous Permuted Kernel Problem]
    \phantom{m}
    \begin{itemize}
        \item[-] Parameters: a prime power $q$ and positive integers $\ell, m, n$.
        \item[-] Instance: a triple of matrices $(\bm{A}, \bm{B}, \bm{C}) \in \mathbb{F}_q^{\ell \times m} \times \mathbb{F}_q^{m \times n} \times \mathbb{F}_q^{\ell \times n}$.
        \item[-] Task: finding a permutation $\pi \in \mathbb{S}_m$ such that $\bm{A}\bm{\pi}\bm{B} = \bm{C}$.
    \end{itemize}
\end{problem}

If $\bm{C} = \bm{0}$ and $n = 1$, then the IPKP corresponds to the version of the PKP formulated in Section~\ref{sec:intro}.
In fact, the PKP is sometimes called \emph{homogeneous} PKP.
Moreover, for $n = 1$ the problem is called \emph{monodimensional}, while for $n > 1$ it is called \emph{multidimensional}.
The heuristic formula for the expected number of solutions naturally generalizes to the multidimensional case, becoming $m! / q^{\ell n}$.

We let $N_\sol(\bm{A}, \bm{B}, \bm{C})$ and $N_\sol(\bm{A}, \bm{B})$ denote the number of solutions to the IPKP instance $(\bm{A}, \bm{B}, \bm{C})$ and to the PKP instance $(\bm{A}, \bm{B})$, respectively.

In order to generate hard instances of the IPKP or the PKP, it is recommended to take $\rank(\bm{A}) = \ell$ and $\rank(\bm{B}) = n$ (see, e.g.,~\cite[Section V.A]{10.1109/TIT.2023.3323068}).
Note that these conditions and the existence of a solution imply that $\max(\ell, n) \leq m$ for the IPKP, and $\ell + n \leq m$ for the PKP.
Hence, a simple algorithm to generate a hard random instance of the IPKP, respectively the PKP, together with a solution of such an instance is provided by $\GenIPKP$ (Figure~\ref{fig:GenIPKP}), respectively $\GenPKP$ (Figure~\ref{fig:GenPKP}).
Indeed, the algorithm used in the implementation of \textsf{PERK}~\cite{PERK_website} is equivalent to $\GenIPKP$, with the only difference that $\bm{A}$ is sampled from $\mathbb{F}_q^{\ell \times m}$ instead of $\mathbb{F}_q^{\ell \times m, \ell}$.

\begin{figure}[h]
    \centering
    \includegraphics{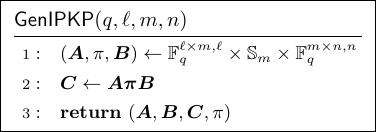}
    \caption{A generator of random instances of the IPKP, assuming that $\max(\ell, n) \leq m$.}
    \label{fig:GenIPKP}
\end{figure}

\begin{figure}[h]
    \centering
    \includegraphics{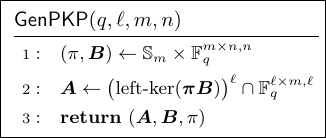}
    \caption{A generator of random instances of the PKP, assuming that $\ell + n \leq m$.}
    \label{fig:GenPKP}
\end{figure}

Furthermore, other recommended conditions for the hardness of the instance are that $\bm{A}$ has pairwise distinct columns and that $\bm{B}$ has pairwise distinct rows~\cite[Section V.A]{10.1109/TIT.2023.3323068}.
We consider only the second condition.
We let $\GenIPKP^\star$ and $\GenPKP^\star$ be defined as $\GenIPKP$ and $\GenPKP$ but with $\bm{B}$ taken from $\mathbb{F}_q^{\star m \times n, n}$.
Note that $\mathbb{F}_q^{\star m \times n, n}$ is nonempty if and only if $n \leq m < q^n$.

\section{Expected number of solutions}\label{sec:number-of-solutions}

Our first result is a formula for the expected number of solutions to a random instance of the IPKP generated by $\GenIPKP$.

\begin{theorem}\label{thm:E-GenIPKP}
    Let $\ell, m, n$ be positive integers with $\max(\ell, n) \leq m$, and let $(\bm{A}, \bm{B}, \bm{C}, \pi) \gets \GenIPKP(q,\ell,m,n)$.
    Then
    \begin{equation}\label{equ:E-GenIPKP-0}
        \mathbb{E}\big[N_{\sol}(\bm{A}, \bm{B}, \bm{C})\big] = \sum_{k \,=\, 1}^m |\mathbb{S}_{m,k}| \!\!
        \sum_{r \,=\, 0}^{\min(\ell, m - k)} \frac{|\mathbb{F}_q^{\ell \times (m-k), r}| \, |\mathbb{F}_q^{(\ell-r) \times k, \ell - r}| \, q^{k r}}{|\mathbb{F}_q^{\ell \times m, \ell}|} \prod_{i \,=\, 0}^{n - 1} \frac{q^{m - r} - q^i}{q^m - q^i}
    \end{equation}
\end{theorem}

Regarding~\eqref{equ:E-GenIPKP-0}, note that $|\mathbb{S}_{m,k}|$ is a \emph{Stirling number of the first kind}, which can be computed recursively (see, e.g., \cite[Eq.\ 6.8]{MR1397498}), while the cardinality of the set of matrices of a prescribed rank is given by an explicit formula (see Lemma~\ref{lem:rank-count}).

Theorem~\ref{thm:E-GenIPKP} shows that, when random instances of the IPKP are generated by $\GenIPKP$, the heuristic formula for the expected number of solutions can be very far off from the true value.
For example, let $q = 1021$, $\ell = 35$, $m = 79$, and $n = 3$, which correspond to the first parameter set of \textsf{PERK} \cite[Table~2]{Bettaieb2024}.
On the one hand, the heuristic formula predicts that the expected number of solutions (not counting $\pi$, which is a solution by construction) is equal to $m! / q^{\ell n} \approx 10^{-199}$.
On the other hand, Theorem~\ref{thm:E-GenIPKP} yields that the true value is about $2.89 \!\cdot\! 10^{-6}$ (again, not counting $\pi$).
Thus the heuristic formula underestimates by a factor of $10^{-193}$ the number of solutions that are not equal to $\pi$.
(However, note that $2.89 \!\cdot\! 10^{-6}$ is still irrelevant and has no impact on the parameter choice of \textsf{PERK}.)

Our second result is a formula for the expected value of the number of solutions to a random instance of the monodimensional IPKP generated by $\GenIPKP^\star$.

\begin{theorem}\label{thm:E-mono-GenIPKP-dist}
    Let $\ell$ and $m$ be positive integers with $\ell \leq m < q$, and let
    $(\bm{A}, \bm{B}, \bm{C}, \pi) \gets \GenIPKP^\star(q,\ell,m,1)$.
    Then
    \begin{equation}\label{equ:E-mono-GenIPKP-dist-0}
        \mathbb{E}\big[N_{\sol}(\bm{A}, \bm{B}, \bm{C})\big] = 1 + \frac{(m! - 1)(q^{m - \ell} - 1)}{q^m - 1} .
    \end{equation}
\end{theorem}

In the right-hand side of~\eqref{equ:E-mono-GenIPKP-dist-0}, the additive term $1$ corresponds to the solution $\pi$, while the second additive term corresponds to additional solutions.
Furthermore, as $m - \ell \to +\infty$, the second term is asymptotic to $m! / q^\ell$, in agreement with the heuristic formula.

Our third result is a formula for the expected number of solutions to a random instance of the monodimensional PKP generated by $\GenPKP$.

\begin{theorem}\label{thm:E-mono-GenPKP}
    Let $\ell, m$ be integers with $0 < \ell < m$.
    \mbox{If $(\bm{A}, \bm{B}, \pi) \gets \GenPKP(q,\ell,m,1)$ then}
    \begin{align}\label{equ:E-mono-GenPKP-0}
        \mathbb{E}&[N_\sol(\bm{A}, \bm{B})]
        = \frac{m!(q^{m-\ell} - q)}{q^m - q} \nonumber\\
        &+ \frac{m!(q^m - q^{m - \ell})}{(q^m - 1)(q^m - q)}\left(\sum_{d \,\mid\, q - 1} \varphi(d) \binom{\lfloor (q+m-1) / d\rfloor}{\lfloor m / d\rfloor} -q + 1\right) .
    \end{align}
\end{theorem}

Although the first additive term on the right-hand side of \eqref{equ:E-mono-GenPKP-0} is asymptotic to $m! / q^\ell$ as $m - \ell \to +\infty$, the second additive term is not negligible.
Hence, the expected number of solutions can be significantly larger than that predicted by the heuristic formula.
For instance, let $q = 251$, $\ell = 41$, and $m = 69$, which correspond to the parameter set for the security level $\lambda = 128$ of \textsf{PKP-DSS} \cite[Table 1]{INDOCRYPT:BFKMPP19}.
Then the heuristic formula predicts about $m! / q^\ell \approx 0.7$ solutions, while Theorem~\ref{thm:E-mono-GenPKP} yields that the true value is about $5412$.

Our forth and final result is a formula for the expected number of solutions to a random instance of the monodimensional PKP generated by $\GenPKP^*$.

\begin{theorem}\label{thm:E-mono-GenPKP-dist}
	Let $\ell$ and $m$ be positive integers with $\ell < m < q$.
	If $(\bm{A}, \bm{B}, \pi) \gets \GenPKP^*(q,\ell,m,1)$ then
    \begin{equation}\label{equ:E-mono-GenPKP-dist-0}
        \mathbb{E}[N_\sol(\bm{A}, \bm{B})]
        = \frac{m!(q^{m - \ell} - q)}{q^m - q} + \frac{q^m - q^{m-\ell}}{(q^m - q)\binom{q-1}{m}} \sum_{d \,\mid\, \gcd(q - 1, m)} \varphi(d) \binom{(q-1)/d}{m/d} .
    \end{equation}
\end{theorem}

As $m - \ell \to +\infty$ (and consequently $q \to +\infty$), the first additive in the right-hand side of~\eqref{equ:E-mono-GenPKP-dist-0} is asymptotic to $m! / q^\ell$, while the second additive term is not exceeding $3$, unless we have that $m = q - 1$ (see Lemma~\ref{lem:sum-phi-binomial-bound}).
Therefore, we can say that Theorem~\ref{thm:E-mono-GenPKP-dist} is in agreement with the heuristic formula.

\section{Proofs}\label{sec:proofs}

\subsection{Preliminaries}

We collect in this section some preliminary lemmas needed later.

\begin{lemma}\label{lem:rank-count}
    Let $m,n,r \geq 0$ be integers.
    Then we have that
    \begin{equation*}
        |\mathbb{F}_q^{m \times n, r}| = \prod_{i \,=\, 0}^{r - 1} \frac{(q^m - q^i)(q^n - q^i)}{q^r - q^i} ,
    \end{equation*}
    with the usual convention that the empty product is equal to $1$.
\end{lemma}
\begin{proof}
    See, e.g., \cite{MR1533848}.
\end{proof}

\begin{lemma}\label{lem:prob-first-columns-rank}
    Let $\ell, m, m_1, m_2 \geq 0$ be integers with $\ell \leq m = m_1 + m_2$.
    Let $\bm{A} \gets \mathbb{F}_q^{\ell \times m, \ell}$ and write $\bm{A} = (\bm{A}_1 \mid \bm{A}_2)$ where $\bm{A}_1 \in \mathbb{F}_q^{\ell \times m_1}$ and $\bm{A}_2 \in \mathbb{F}_q^{\ell \times m_2}$.
    Then
    \begin{equation*}
        \mathbb{P}\big[\rank(\bm{A}_1) = r\big] = \frac{|\mathbb{F}_q^{\ell \times m_1, r}| \, |\mathbb{F}_q^{(\ell-r) \times m_2, \ell - r}| \, q^{m_2 r}}{|\mathbb{F}_q^{\ell \times m, \ell}|} ,
    \end{equation*}
    for every integer $r$ with $0 \leq r \leq \min(\ell, m_1)$.
\end{lemma}
\begin{proof}
    Let us count the number of $\bm{A} = (\bm{A}_1 \mid \bm{A}_2) \in \mathbb{F}_q^{\ell \times m, \ell}$ such that $\rank(\bm{A}_1) = r$.
    Clearly, there are $|\mathbb{F}_q^{\ell \times m_1, r}|$ choices for $\bm{A}_1$.
    We have to determine the number of possible choices for $\bm{A}_2$.
    Let $C$ and $C_1$ be the columnspaces of $\bm{A}$ and $\bm{A}_1$, respectively.
    Hence, we have that $\dim(C) = \ell$, $\dim(C_1) = r$, and $C_1 \subseteq C$.
    Consequently, the quotient space $C / C_1$ can be identified with $\mathbb{F}_q^{(\ell - r) \times 1}$ via the choice of a basis.
    Let $\rho \colon C \to \mathbb{F}_q^{(\ell - r) \times 1}$ be the natural projection onto the quotient space, and let $\rho(\bm{A}_2)$ be the $(\ell - r) \times m_2$ matrix obtained by applying $\rho$ to each column of $\bm{A}_2$.
    Since the columns of $\bm{A}$ generates $C$, it follows easily that the columns of $\rho(\bm{A}_2)$ generates $\mathbb{F}_q^{(\ell - r) \times 1}$. Hence, we have that $\rank(\rho(\bm{A}_2)) = \ell - r$.
    Moreover, since the kernel of $\rho$ has dimension $r$, for each choice of $\rho(\bm{A}_2) \in \mathbb{F}_q^{(\ell - r) \times m_2, \ell - r}$ there correspond $q^{m_2 r}$ choices of $\bm{A}_2 \in \mathbb{F}_q^{\ell \times m_2}$.
    Thus we get that there are $|\mathbb{F}_q^{(\ell-r) \times m_2, \ell - r}| \, q^{m_2 r}$ choices for $\bm{A}_2$.
    The main claim follows.
\end{proof}

\begin{lemma}\label{lem:prob-zero-products}
    Let $\ell,m,m^\prime,n,s \geq 0$ be integers such that $s \leq \min(m, m^\prime)$, $\ell \leq m$, and $n \leq m^\prime$, and let $\bm{M} \in \mathbb{F}_q^{m \times m^\prime, s}$.
    If $(\bm{A}, \bm{B}) \gets \mathbb{F}_q^{\ell \times m, \ell} \times \mathbb{F}_q^{m^\prime \times n, n}$ then
    \begin{enumerate}
        \item\label{ite:prob-zero-products-1} $\mathbb{P}[\bm{A}\bm{M} = \bm{0}] = \prod_{i = 0}^{\ell - 1} \frac{q^{m - s} - q^i}{q^{m} - q^i}$;
        \item\label{ite:prob-zero-products-2} $\mathbb{P}[\bm{M}\bm{B} = \bm{0}] = \prod_{i = 0}^{n - 1} \frac{q^{m^\prime - s} - q^i}{q^{m^\prime} - q^i}$;
        \item\label{ite:prob-zero-products-3} $\mathbb{P}[\bm{A}\bm{M}\bm{B} = \bm{0}] = \sum_{r = 0}^{\min(\ell, s)} \frac{|\mathbb{F}_q^{\ell \times s, r}| \, |\mathbb{F}_q^{(\ell-r) \times (m-s), \ell - r}| \, q^{(m-s) r}}{|\mathbb{F}_q^{\ell \times m, \ell}|} \prod_{i = 0}^{n - 1} \frac{q^{m^\prime - r} - q^i}{q^{m^\prime} - q^i}$.
    \end{enumerate}
\end{lemma}
\begin{proof}
    First, note that \ref{ite:prob-zero-products-1} follows from \ref{ite:prob-zero-products-2} since $\bm{A}\bm{M} = \bm{0}$ is equivalent to $\bm{M}^\transpose \bm{A}^\transpose = \bm{0}$ and transposition does not change the rank.

    Let us prove \ref{ite:prob-zero-products-2}.
    We have that $\bm{M}\bm{B} = \bm{0}$ is equivalent to each column of $\bm{B}$ belonging to the kernel of $\bm{M}$, which has dimension $m^\prime - s$.
    Thus there are $|\mathbb{F}_q^{(m^\prime - s) \times n, n}|$ possible choices for $\bm{B}$.
    Hence, using Lemma~\ref{lem:rank-count}, we get that
    \begin{equation*}
        \mathbb{P}[\bm{M}\bm{B} = \bm{0}]
        = \frac{|\mathbb{F}_q^{(m^\prime - s) \times n, n}|}{|\mathbb{F}_q^{m^\prime \times n, n}|}
        = \prod_{i \,=\, 0}^{n - 1} \frac{q^{m^\prime - s} - q^i}{q^{m^\prime} - q^i}
    \end{equation*}
    which proves \ref{ite:prob-zero-products-2}.

    It remains to prove \ref{ite:prob-zero-products-3}.
    Since $\rank(\bm{M}) = s$, there exist $\bm{Q} \in \mathbb{F}_q^{m \times m, m}$ and $\bm{R} \in \mathbb{F}_q^{s \times m^\prime, s}$ such that
    \begin{equation*}
        \bm{M} = \bm{Q} \begin{pmatrix} \bm{R} \\ \bm{0} \end{pmatrix} .
    \end{equation*}
    Let $\widetilde{\bm{A}} \gets \mathbb{F}_q^{\ell \times m, \ell}$ and write $\widetilde{\bm{A}} = (\widetilde{\bm{A}}_1 \mid \widetilde{\bm{A}}_2)$ where $\widetilde{\bm{A}}_1 \in \mathbb{F}_q^{\ell \times s}$ and $\widetilde{\bm{A}}_2 \in \mathbb{F}_q^{\ell \times (m-s)}$.
    Since $\bm{A}\bm{Q} \eqdistr \widetilde{\bm{A}}$ and $\rank(\widetilde{\bm{A}}_1\bm{R}) = \rank(\widetilde{\bm{A}}_1) \leq \min(\ell, s)$, we have that
    \begin{align}\label{equ:prob-zero-products-1}
        \mathbb{P}[\bm{A}\bm{M}\bm{B} = \bm{0}]
        &= \mathbb{P}\big[\bm{A}\bm{Q} \!\begin{pmatrix} \bm{R} \\ \bm{0} \end{pmatrix}\! \bm{B} = \bm{0}\big]
        = \mathbb{P}\big[\widetilde{\bm{A}} \!\begin{pmatrix} \bm{R} \\ \bm{0} \end{pmatrix}\! \bm{B} = \bm{0}\big]
        = \mathbb{P}[\widetilde{\bm{A}}_1\bm{R}\bm{B} = \bm{0}] \nonumber\\
        &= \sum_{r \,=\, 0}^{\min(\ell, s)} \sum_{\bm{N} \,\in\, \mathbb{F}_q^{\ell \times m^\prime\!, r}} \mathbb{P}[\widetilde{\bm{A}}_1\bm{R} = \bm{N}] \, \mathbb{P}[\bm{N}\bm{B} = \bm{0}] .
    \end{align}
    By \ref{ite:prob-zero-products-2}, we have that
    \begin{equation}\label{equ:prob-zero-products-2}
        \mathbb{P}[\bm{N}\bm{B} = \bm{0}] = \prod_{i \,=\, 0}^{n - 1} \frac{q^{m^\prime - r} - q^i}{q^{m^\prime} - q^i} ,
    \end{equation}
    while, using again the fact that $\rank(\widetilde{\bm{A}}_1\bm{R}) = \rank(\widetilde{\bm{A}}_1)$ and employing Lemma~\ref{lem:prob-first-columns-rank}, we get that
    \begin{align}\label{equ:prob-zero-products-3}
        \sum_{\bm{N} \,\in\, \mathbb{F}_q^{\ell \times m^\prime\!, r}}&   \mathbb{P}[\widetilde{\bm{A}}_1\bm{R} = \bm{N}]
        = \mathbb{P}[\rank(\widetilde{\bm{A}}_1\bm{R}) = r] \nonumber\\
        &= \mathbb{P}[\rank(\widetilde{\bm{A}}_1) = r]
        = \frac{|\mathbb{F}_q^{\ell \times s, r}| \, |\mathbb{F}_q^{(\ell-r) \times (m-s), \ell - r}| \, q^{(m-s) r}}{|\mathbb{F}_q^{\ell \times m, \ell}|} .
    \end{align}
    Thus, putting together \eqref{equ:prob-zero-products-1}, \eqref{equ:prob-zero-products-2}, and \eqref{equ:prob-zero-products-3}, we obtain \ref{ite:prob-zero-products-3}.
\end{proof}

For each permutation $\sigma \in \mathbb{S}_m$ and for each $\lambda \in \mathbb{F}_q^*$, let $E_{\sigma, \lambda}$ be the set of $\bm{x} \in \mathbb{F}_q^{m \times 1, 1}$ such that $\bm{\sigma}\bm{x} = \lambda \bm{x}$.
Moreover, let $E_\sigma := \bigcup_{\lambda \in \mathbb{F}_q^*} E_{\sigma, \lambda}$.
In other words, we have that $E_{\sigma, \lambda}$ is the set of eigenvectors of $\bm{\sigma}$ with eigenvalue $\lambda$, if $\lambda$ is an eigenvalue of $\bm{\sigma}$, otherwise $E_{\sigma, \lambda} = \varnothing$.
(Note that $\lambda = 0$ is not an eigenvalue of $\bm{\sigma}$, since $\bm{\sigma}$ is invertible.)
Furthermore, let $E_{\sigma, \lambda}^\star := E_{\sigma, \lambda} \cap \mathbb{F}_q^{\star m \times 1, 1}$ and $E_\sigma^\star := \bigcup_{\lambda \in \mathbb{F}_q^*} E_{\sigma, \lambda}^\star$.

\begin{lemma}\label{lem:sum-E-sigma}
    For each positive integer $m$, we have that
    \begin{equation}\label{equ:sum-E-sigma-0}
        \sum_{\sigma \,\in\, \mathbb{S}_m} |E_\sigma| = m! \left(\sum_{d \,\mid\, q - 1} \varphi(d) \binom{\lfloor (q + m - 1)/d \rfloor}{\lfloor m/d \rfloor }  - q + 1\right) .
    \end{equation}
\end{lemma}
\begin{proof}
    Let $\sigma \in \mathbb{S}_m$, $\lambda \in \mathbb{F}_q^*$, and $\bm{x} \in \mathbb{F}_q^{m \times 1, 1}$.
    Write $\sigma = \gamma_1 \cdots \gamma_t$ and $\bm{x} = (x_1, \dots, x_m)^\transpose$, where $\gamma_1, \dots, \gamma_t$ are disjoint cycles and $x_1, \dots, x_m \in \mathbb{F}_q$.
    By ``$\bm{x}$ over $\gamma_k$'' we mean the sequence $(x_j)$ where $j$ runs over the orbit of the elements moved by $\gamma_k$.
    Also, we let $|\gamma_k|$ denote the length of the cycle $\gamma_k$.
    It follows easily that $\bm{x} \in E_{\sigma, \lambda}$ if and only if, for each $k \in \{1, \dots, t\}$, we have that either $\lambda^{|\gamma_k|} = 1$ and $\bm{x}$ over $\gamma_k$ is a geometric progression of ratio $\lambda$, or $\lambda^{|\gamma_k|} \neq 1$ and $\bm{x}$ over $\gamma_k$ is a zero sequence.
    Hence, we have that
    \begin{equation*}
        |E_{\sigma, \lambda}| = q^{|\{k \,:\, \lambda^{|\gamma_k|} \,=\, 1\}|} - 1 .
    \end{equation*}
    Therefore, since the $E_\sigma = \bigcup_{\lambda \in \mathbb{F}_q^*} E_{\sigma, \lambda}$ is a disjoint union, we get that
    \begin{equation}\label{equ:sum-E-sigma-1}
        |E_\sigma| = \sum_{\lambda \,\in\, \mathbb{F}_q^*} \left(q^{|\{k \,:\, \lambda^{|\gamma_k|} \,=\, 1\}|} - 1\right) = \sum_{d \,\mid\, q - 1} \varphi(d) \, q^{|\{k \,:\, d \text{ divides } |\gamma_k|\}|} - q + 1 ,
    \end{equation}
    where we used the facts that: $\lambda^{|\gamma_k|} \,=\, 1$ if and only if the multiplicative order of $\lambda$ divides $|\gamma_k|$; and there are $\varphi(d)$ elements of order $d$ in the multiplicative (cyclic) group $\mathbb{F}_q^*$.

    At this point, we employ the theory of combinatorial classes and their generating functions \cite[Part A, Chapters I and II]{MR2483235}.
    Letting
    \begin{equation*}
        F_d(z) := \log\!\left(\frac1{1 - z}\right) - \frac1{d} \log\!\left(\frac1{1 - z^d}\right)
        \quad\text{ and }\quad
        G_d(z) := \frac1{d}\log\!\left(\frac1{1 - z^d}\right) ,
    \end{equation*}
    we have that $F_d(z)$, respectively $G_d(z)$, is the exponential generating function of the combinatorial class of cycles with length not divisible, respectively divisible, by $d$.
    Hence, since each permutation can be uniquely written as a product of disjoint cycles, we have that
    \begin{align*}
        \sum_{m \,=\, 0}^\infty &\left(\sum_{\sigma \,=\, \gamma_1 \cdots\, \gamma_k \,\in\, \mathbb{S}_m} q^{|\{k \,:\, d \text{ divides } |\gamma_k|\}|}\right) \frac{z^m}{m!}
        = \exp\!\big(F_d(z) + G_d(z) q\big) \\
        &= \frac1{1 - z} \cdot \frac1{(1 - z^d)^{(q - 1)/d}}
        = \sum_{i \,=\, 0}^\infty z^i \sum_{j \,=\, 0}^\infty \binom{(q - 1)/d + j - 1}{j} z^{dj} \\
        &= \sum_{m \,=\, 0}^\infty \sum_{j \,=\, 0}^{\lfloor m / d\rfloor} \binom{(q - 1)/d + j - 1}{j} z^m
        = \sum_{m \,=\, 0}^\infty \binom{\lfloor (q + m - 1)/d \rfloor}{ \lfloor m / d\rfloor} z^m ,
    \end{align*}
    where we used the identity $\sum_{j=0}^r \binom{s + j - 1}{j} = \binom{s + r}{r}$, which holds for all integers $r, s \geq 0$.
    Therefore, we get that
    \begin{equation}\label{equ:sum-E-sigma-2}
        \sum_{\sigma \,=\, \gamma_1 \cdots\, \gamma_k \,\in\, \mathbb{S}_m} q^{|\{k \,:\, d \text{ divides } |\gamma_k|\}|} = m! \binom{\lfloor (q + m - 1)/d \rfloor}{ \lfloor m / d\rfloor} .
    \end{equation}
    Finally, summing \eqref{equ:sum-E-sigma-1} over $\sigma \in \mathbb{S}_m$, and employing \eqref{equ:sum-E-sigma-2}, we obtain \eqref{equ:sum-E-sigma-0}, as desired.
\end{proof}

\begin{lemma}\label{lem:sum-E-sigma-dist}
    For each positive integer $m$, we have that
    \begin{equation}\label{lem:sum-E-sigma-dist-0}
        \sum_{\sigma \,\in\, \mathbb{S}_m} |E_\sigma^\star| = m! \sum_{d \,\mid\, \gcd(q - 1, m)} \varphi(d) \binom{(q-1)/d}{m/d} .
    \end{equation}
\end{lemma}
\begin{proof}
    The proof is similar to that of Lemma~\ref{lem:sum-E-sigma}.
    Let $\sigma \in \mathbb{S}_m$, $\lambda \in \mathbb{F}_q^*$, and $\bm{x} \in \mathbb{F}_q^{\star m \times 1, 1}$.
    Write $\sigma = \gamma_1 \cdots \gamma_t$, where $\gamma_1, \dots, \gamma_t$ are disjoint cycles, and let $d$ be the multiplicative order of $\lambda$.
    It follows easily that $\bm{x} \in E_{\sigma, \lambda}^\star$ if and only if: for each $k \in \{1,\dots, t\}$ we have that $|\gamma_k| = d$ and $\bm{x}$ over $\gamma_k$ is a nonzero geometric progression of ratio $\lambda$; and the aforementioned geometric progressions are disjoint.
    In other words, we have that $\bm{x} \in E_{\sigma, \lambda}^\star$ if and only if $\bm{x}$ over $\gamma_k$, for $k \in \{1, \dots, t\}$, are disjoint cosets of the quotient group $\mathbb{F}_q^* / \langle \lambda\rangle$.
    Consequently, we get that $|E_{\sigma, \lambda}^\star| = 0$ if $|\gamma_k| \neq d$ for some $k$, while
    \begin{equation}\label{equ:sum-E-sigma-dist-1}
        |E_{\sigma, \lambda}^\star| = \binom{(q - 1)/d}{m / d} (m / d)! \cdot d^{m / d}
    \end{equation}
    if $|\gamma_k| = d$ for each $k \in \{1, \dots, t\}$ (and so $t = m/d$).
    More precisely, the first factor in \eqref{equ:sum-E-sigma-dist-1} is equal to the number of choices of one of the $(q-1)/d$ classes of $\mathbb{F}_q^* / \langle \lambda\rangle$ for each of the $m/d$ cycles of $\sigma$, while the second factor is equal to the number of choices of one of the $d$ elements of each class for each of the $m/d$ cycles.
    Therefore, since $E_\sigma^\star = \bigcup_{\lambda \in \mathbb{F}_q^*} E_{\sigma, \lambda}^\star$ is a disjoint union, we get that
    \begin{equation}\label{equ:sum-E-sigma-dist-2}
        |E_\sigma^\star| = \sum_{d \,\mid\, \gcd(q - 1, m)} \varphi(d) \binom{(q - 1)/d}{m / d} (m / d)! \, d^{m / d} \, \mathbbm{1}\!\big[|\gamma_k| = d \text{ for } k =1,\dots,t\big] ,
    \end{equation}
    where we used the fact that there are $\varphi(d)$ elements of order $d$ in the multiplicative (cyclic) group $\mathbb{F}_q^*$.
    If $d$ divides $m$, then there are $m! / \big(d^{m/d} (m/d)!\big)$ permutations whose cycles have all length equal to $d$ \cite[Eq.~(1.2)]{MR1824028}.
    Therefore, summing \eqref{equ:sum-E-sigma-dist-2} over $\sigma \in \mathbb{S}_m$ we get \eqref{lem:sum-E-sigma-dist-0}, as desired.
\end{proof}

\begin{lemma}\label{lem:sum-phi-binomial-bound}
	For each positive integer $m \leq q - 2$, we have that
	\begin{equation*}
		\sum_{d \,\mid\, \gcd(q - 1, m)} \varphi(d) \binom{(q-1)/d}{m/d} < 3 \binom{q-1}{m} .
	\end{equation*}
\end{lemma}
\begin{proof}
	If $m = q - 2$ then the claim is obvious since $\gcd(q - 1, m) = 1$.
	Hence, assume that $m \leq q - 3$.
	Note that $\binom{a / d}{b / d}^d \leq \binom{a}{b}$ for all positive integers $a,b,d$ such that $d \mid \gcd(a, b)$.
	Indeed, this claim follows by considering that we can pick $b$ elements from $\{1, \dots, a\}$ by picking $b / d$ elements from $\{a(k-1)/d+1,\dots, ak/d\}$ for each $k \in \{1,\dots,d\}$.
	Consequently, we have that
	\begin{align*}
		\sum_{\substack{d \,\mid\, \gcd(q - 1, m) \\ d \,>\, 1}} &\varphi(d) \binom{(q-1)/d}{m/d}
		\leq \sum_{\substack{d \,\mid\, \gcd(q - 1, m) \\ d \,>\, 1}} \varphi(d) \binom{q-1}{m}^{1/d} \\
		&\leq \sum_{\substack{d \,\mid\, \gcd(q - 1, m) \\ d \,>\, 1}} \varphi(d) \binom{q-1}{m}^{1/2}
		= \gcd(q - 1, m)\binom{q-1}{m}^{1/2}
		\leq m \binom{q-1}{m}^{1/2} ,
	\end{align*}
	where we employed the well-known identity $\sum_{d \mid n} \varphi(d) = n$, which holds for every positive integer $n$.
	Therefore, we get that
	\begin{equation*}
		\binom{q-1}{m}^{\!\!-1} \sum_{d \,\mid\, \gcd(q - 1, m)} \varphi(d) \binom{(q-1)/d}{m/d} \leq 1 + m \binom{q-1}{m}^{\!\!-1/2} \leq 1 + m \binom{m + 2}{m}^{\!\!-1/2} < 3 ,
	\end{equation*}
	as desired.
\end{proof}

\subsection{Proof of Theorem~\ref{thm:E-GenIPKP}}

Let $\ell, m, n$ be positive integers with $\max(\ell, n) \leq m$, and let
\begin{equation*}
    (\bm{A}, \bm{B}, \bm{C}, \pi) \gets \GenIPKP(q,\ell,m,n) .
\end{equation*}
We have that
\begin{align}\label{equ:N-sol-IPKP-1}
    N_\sol&(\bm{A}, \bm{B}, \bm{C})
    = \sum_{\rho \,\in\, \mathbb{S}_m} \mathbbm{1}[\bm{A}\bm{\rho}\bm{B} = \bm{C}]
    = \sum_{\rho \,\in\, \mathbb{S}_m} \mathbbm{1}[\bm{A}\bm{\rho}\bm{B} = \bm{A}\bm{\pi}\bm{B}]
    = \sum_{\rho \,\in\, \mathbb{S}_m} \mathbbm{1}[\bm{A}(\bm{\rho} - \bm{\pi})\bm{B} = \bm{0}] \nonumber \\
    &= \sum_{\sigma \,\in\, \mathbb{S}_m} \mathbbm{1}[\bm{A}(\bm{\sigma} - \bm{I})\bm{\pi}\bm{B} = \bm{0}]
    \eqdistr \sum_{\sigma \,\in\, \mathbb{S}_m} \mathbbm{1}[\bm{A}(\bm{\sigma} - \bm{I})\widetilde{\bm{B}} = \bm{0}] ,
\end{align}
where $\widetilde{\bm{B}} \gets \mathbb{F}_q^{m \times n, n}$ and we employed the substitution $\rho = \sigma\pi$ and the fact that $\widetilde{\bm{B}} \eqdistr \bm{\pi}\bm{B}$.
Consequently, we get that
\begin{equation}\label{equ:E-IPKP-1}
    \mathbb{E}\big[N_\sol(\bm{A}, \bm{B}, \bm{C})\big]
    = \sum_{\sigma \,\in\, \mathbb{S}_m} \mathbb{P}[\bm{A}(\bm{\sigma} - \bm{I})\widetilde{\bm{B}} = \bm{0}] .
\end{equation}
Note that $\bm{A}$ and $\widetilde{\bm{B}}$ are independent random variables.
If $k \leq m$ is a positive integer and $\sigma \in \mathbb{S}_{m,k}$, then it is easy to prove that $\rank(\bm{\sigma} - \bm{I}) = m - k$.
Therefore, from \eqref{equ:E-IPKP-1} and Lemma~\ref{lem:prob-zero-products}\ref{ite:prob-zero-products-3}, we have that
\begin{align*}
    \mathbb{E}&\big[N_\sol(\bm{A}, \bm{B}, \bm{C})\big]
    = \sum_{k \,=\, 1}^m \, \sum_{\sigma \,\in\, \mathbb{S}_{m,k}} \mathbb{P}[\bm{A}(\bm{\sigma} - \bm{I})\widetilde{\bm{B}} = \bm{0}] \\
    &= \sum_{k \,=\, 1}^m |\mathbb{S}_{m,k}|
    \!\!\sum_{r \,=\, 0}^{\min(\ell, m - k)} \frac{|\mathbb{F}_q^{\ell \times (m-k), r}| \, |\mathbb{F}_q^{(\ell-r) \times k, \ell - r}| \, q^{k r}}{|\mathbb{F}_q^{\ell \times m, \ell}|} \prod_{i \,=\, 0}^{n - 1} \frac{q^{m - r} - q^i}{q^m - q^i} ,
\end{align*}
as desired.
The proof is complete.

\subsection{Proof of Theorem~\ref{thm:E-mono-GenIPKP-dist}}

Let $\ell$ and $m$ be integers with $0 \leq \ell \leq m < q$, and let $(\bm{A}, \bm{B}, \bm{C}, \pi) \gets \GenIPKP^\star(q,\ell,m,1)$.
By reasoning as in the beginning of the proof of Theorem~\ref{thm:E-GenIPKP}, we obtain that
\begin{equation}\label{equ:E-mono-IPKP-dist-1}
    \mathbb{E}\big[N_\sol(\bm{A}, \bm{B}, \bm{C})\big]
    = \sum_{\sigma \,\in\, \mathbb{S}_m} \mathbb{P}[\bm{A}(\bm{\sigma} - \bm{I})\widetilde{\bm{B}} = \bm{0}] ,
\end{equation}
where $\widetilde{\bm{B}} \gets \mathbb{F}_q^{\star m \times 1, 1}$.
It follows easily that, for each $\sigma \in \mathbb{S}_m$ and $\bm{x} \in \mathbb{F}_q^{\star m \times 1, 1}$, the rank of $(\bm{\sigma} - \bm{I})\bm{x}$ is equal to $0$ if $\sigma$ is the identity, and is equal to $1$ otherwise.
Hence, since $\bm{A}$ and $\widetilde{\bm{B}}$ are independent random variables, from Lemma~\ref{lem:prob-zero-products}\ref{ite:prob-zero-products-1} we get that
\begin{equation}\label{equ:E-mono-IPKP-dist-2}
    \mathbb{P}[\bm{A}(\bm{\sigma} - \bm{I})\widetilde{\bm{B}} = \bm{0}]
    = \begin{cases}
        1 & \text{ if } \bm{\sigma} = \bm{I}; \\
        (q^{m - \ell} - 1) / (q^m - 1) & \text{ if } \bm{\sigma} \neq \bm{I}. \\
    \end{cases}
\end{equation}
Therefore, from~\eqref{equ:E-mono-IPKP-dist-1} and~\eqref{equ:E-mono-IPKP-dist-2} we obtain~\eqref{equ:E-mono-GenIPKP-dist-0}, as desired.
The proof is complete.

\subsection{Proof of Theorem~\ref{thm:E-mono-GenPKP}}

Let $\ell$ and $m$ be positive integers with $\ell < m$, and let
\begin{equation*}
	(\bm{A}, \bm{B}, \pi) \gets \GenPKP(q,\ell,m,1)
	\quad\text{ and }\quad
	(\widetilde{\bm{A}}, \widetilde{\bm{B}}, \widetilde{\bm{C}}, \widetilde{\pi}) \gets \GenIPKP(q,\ell,m,1) .
\end{equation*}
Since line~2 of $\GenPKP$ is equivalent to a loop that keeps picking $\bm{A} \gets \mathbb{F}_q^{\ell \times m, \ell}$ until $\bm{A}\bm{\pi}\bm{B} = \bm{0}$ (essentially, a rejection sampling), it follows that the distribution of $(\bm{A}, \bm{B})$ is equal to the conditional distribution of $(\widetilde{\bm{A}}, \widetilde{\bm{B}})$ with respect to the event $\widetilde{\bm{C}} = \bm{0}$.

Hence, we have that
\begin{align}\label{equ:E-mono-GenPKP-1}
	\mathbb{E}&[N_\sol(\bm{A}, \bm{B})]
	= \sum_{\rho \,\in\, \mathbb{S}_m} \mathbb{P}[\widetilde{\bm{A}}\bm{\rho}\widetilde{\bm{B}} = \bm{0} \mid \widetilde{\bm{C}} = \bm{0}] \nonumber \\
	&= \frac1{\mathbb{P}[\widetilde{\bm{C}} = \bm{0}]} \sum_{\rho \,\in\, \mathbb{S}_m} \mathbb{P}[\widetilde{\bm{A}}\bm{\rho}\widetilde{\bm{B}} = \bm{0} \text{ and } \widetilde{\bm{C}} = \bm{0}] .
\end{align}
Let $\widehat{\bm{B}} \gets \mathbb{F}_q^{m \times 1, 1}$, so that $\widehat{\bm{B}} \eqdistr \widetilde{\pi}\widetilde{\bm{B}}$.
Since $\widetilde{\bm{A}}$ and $\widehat{\bm{B}}$ are independent random variables, by Lemma~\ref{lem:prob-zero-products}\ref{ite:prob-zero-products-1} we have that
\begin{equation}\label{equ:E-mono-GenPKP-2}
	\mathbb{P}[\widetilde{\bm{C}} = \bm{0}]
	= \mathbb{P}[\widetilde{\bm{A}}\widetilde{\pi}\widetilde{\bm{B}} = \bm{0}]
	= \mathbb{P}[\widetilde{\bm{A}}\widehat{\bm{B}} = \bm{0}]
	= \frac{q^{m-\ell} - 1}{q^m - 1} .
\end{equation}
Moreover, we have that
\begin{align}\label{equ:E-mono-GenPKP-3}
	\sum_{\rho \,\in\, \mathbb{S}_m}& \mathbb{P}[\widetilde{\bm{A}}\bm{\rho}\widetilde{\bm{B}} = \bm{0} \text{ and } \widetilde{\bm{C}} = \bm{0}]
    = \sum_{\rho \,\in\, \mathbb{S}_m} \mathbb{P}[\widetilde{\bm{A}}\bm{\rho}\widetilde{\bm{B}} = \bm{0} \text{ and } \widetilde{\bm{A}}\widetilde{\bm{\pi}}\widetilde{\bm{B}} = \bm{0}] \nonumber\\
	&= \sum_{\sigma \,\in\, \mathbb{S}_m} \mathbb{P}[\widetilde{\bm{A}}\bm{\sigma}\widetilde{\bm{\pi}}\widetilde{\bm{B}} = \bm{0} \text{ and } \widetilde{\bm{A}}\widetilde{\bm{\pi}}\widetilde{\bm{B}} = \bm{0}]
	= \sum_{\sigma \,\in\, \mathbb{S}_m} \mathbb{P}[\widetilde{\bm{A}}\bm{\sigma}\widehat{\bm{B}} = \bm{0} \text{ and } \widetilde{\bm{A}}\widehat{\bm{B}} = \bm{0}] \nonumber\\
	&= \sum_{\sigma \,\in\, \mathbb{S}_m} \mathbb{P}[\widetilde{\bm{A}}(\bm{\sigma}\widehat{\bm{B}} \mid \widehat{\bm{B}}) = \bm{0}] ,
\end{align}
where we employed the substitution $\rho = \sigma\widetilde{\pi}$.

For each $\sigma \in \mathbb{S}_m$ and $\bm{x} \in \mathbb{F}_q^{m \times 1, 1}$, we have that $\rank(\bm{\sigma}\bm{x} \mid \bm{x})$ is equal to $1$ if $\bm{x} \in E_\sigma$, and is equal to $2$ otherwise.
Hence, by Lemma~\ref{lem:prob-zero-products}\ref{ite:prob-zero-products-1}, we get that
\begin{equation*}
	\mathbb{P}[\widetilde{\bm{A}}(\bm{\sigma}\bm{x} \mid \bm{x}) = \bm{0}] =
	\frac{q^{m - \ell} - 1}{q^m - 1}
	\begin{cases}
		1 & \text{ if } \bm{x} \in E_\sigma ; \\
		(q^{m-\ell} - q) / (q^m - q) & \text{ if } \bm{x} \notin E_\sigma .
	\end{cases}
\end{equation*}
Consequently, we obtain that
\begin{align}\label{equ:E-mono-GenPKP-4}
	\mathbb{P}[\widetilde{\bm{A}}(\bm{\sigma}\widehat{\bm{B}} \mid \widehat{\bm{B}}) = \bm{0}]
	&= \mathbb{P}[\widehat{\bm{B}} \in E_\sigma] \, \mathbb{P}\big[\widetilde{\bm{A}}(\bm{\sigma}\widehat{\bm{B}} \mid \widehat{\bm{B}}) = \bm{0} \mid \widehat{\bm{B}} \in E_\sigma \big] \nonumber\\
    &\phantom{mmm}+ \mathbb{P}[\widehat{\bm{B}} \notin E_\sigma] \, \mathbb{P}\big[\widetilde{\bm{A}}(\bm{\sigma}\widehat{\bm{B}} \mid \widehat{\bm{B}}) = \bm{0} \mid \widehat{\bm{B}} \notin E_\sigma \big] \nonumber\\
    &= \frac{q^{m - \ell} - 1}{q^m - 1} \left(\frac{|E_\sigma|}{q^m - 1}
    + \left(1 - \frac{|E_\sigma|}{q^m - 1}\right) \frac{q^{m - \ell} - q}{q^m - q} \right) \nonumber\\
    &= \frac{q^{m - \ell} - 1}{q^m - 1} \left(\frac{q^{m - \ell} - q}{q^m - q} + \frac{q^m - q^{m-\ell}}{(q^m - 1)(q^m - q)} \, |E_\sigma| \right)
\end{align}
Therefore, from \eqref{equ:E-mono-GenPKP-1}, \eqref{equ:E-mono-GenPKP-2}, \eqref{equ:E-mono-GenPKP-3}, and \eqref{equ:E-mono-GenPKP-4}, we have that
\begin{equation*}
	\mathbb{E}[N_\sol(\bm{A}, \bm{B})] = \frac{m!(q^{m - \ell} - q)}{q^m - q} + \frac{q^m - q^{m-\ell}}{(q^m - 1)(q^m - q)} \sum_{\sigma \,\in\, \mathbb{S}_m} |E_\sigma| .
\end{equation*}
At this point, the claim follows from Lemma~\ref{lem:sum-E-sigma}.
The proof is complete.

\subsection{Proof of Theorem~\ref{thm:E-mono-GenPKP-dist}}

Let $\ell$ and $m$ be positive integers with $\ell < m < q$, and let
\begin{equation*}
    (\bm{A}, \bm{B}, \pi) \gets \GenPKP^*(q,\ell,m,1)
    \quad\text{ and }\quad
    (\widetilde{\bm{A}}, \widetilde{\bm{B}}, \widetilde{\bm{C}}, \widetilde{\pi}) \gets \GenIPKP^*(q,\ell,m,1) .
\end{equation*}
The proof proceeds \emph{mutatis mutandis} as the proof of Theorem~\ref{thm:E-mono-GenPKP}.
The only differences are that: we pick $\widehat{\bm{B}} \gets \mathbb{F}_q^{\star m \times 1, 1}$, we have $E_\sigma^\star$ instead of $E_\sigma$ (so that, in particular, it holds $\mathbb{P}[\widehat{\bm{B}} \in E_\sigma^\star] = |E_\sigma^\star| / (m!\binom{q-1}{m})$), and we apply Lemma~\ref{lem:sum-E-sigma-dist} instead of Lemma~\ref{lem:sum-E-sigma}.

\section{Concluding remarks}\label{sec:remarks}

We proved exact formulas for the expected number of solutions to random instances of the IPKP and the PKP generated by the natural algorithms $\GenIPKP$, $\GenIPKP^\star$, $\GenPKP$, and $\GenPKP^\star$.
In addition, we compared these exact formulas with the heuristic formula that has been extensively used in previous works.
For instances generated by $\GenIPKP$ and $\GenPKP$, we showed that the heuristic formula can be far off from the exact values, while for instances generated by $\GenIPKP^\star$ and $\GenPKP^\star$ the heuristic formula provides good approximations to the exact values.

Let us briefly outline some possible directions for future research.
First, it might be interesting to generalize Theorems~\ref{thm:E-mono-GenIPKP-dist}, \ref{thm:E-mono-GenPKP}, and \ref{thm:E-mono-GenPKP-dist} to the multidimensional case $n > 1$.
For Theorem~\ref{thm:E-mono-GenPKP}, respectively Theorem~\ref{thm:E-mono-GenPKP-dist}, this seems related to the study of the rank of matrices of the form $(\bm{\sigma}\bm{X} \mid \bm{X})$, where $\sigma \in \mathbb{S}_m$ and $\bm{X} \in \mathbb{F}_q^{m \times n, n}$, respectively $\bm{X} \in \mathbb{F}_q^{\star m \times n, n}$.

Second, in light of the last paragraph of Section~\ref{sec:IPKP}, one could study the expected number of solutions to random instances of the IPKP and the PKP when $\bm{A}$ has pairwise distinct columns and $\bm{B}$ has pairwise distinct rows.

Third, one could try to compute, or just upper bound, the second moment, or more generally the higher moments, of the number of solutions $N_\sol$ to the random instances of the IPKP or the PKP.
This could make it possible to obtain upper bounds for the tail probability $\mathbb{P}[N_\sol \geq N]$, where $N \geq 2$ is an integer, which are better than bounds based only on the expected value of $N_\sol$, such as the Markov inequality.

Finally, from a more abstract perspective, one could investigate if an
appropriate normalization of $N_\sol$ converges in law to some known random variable (for example, a standard normal random variable) when $\ell, m, n$ go to infinity in a way to be made precise.


\begin{thebibliography}{10}

    \bibitem{C:BCCG92}
    T.~Baritaud, M.~Campana, P.~Chauvaud, and H.~Gilbert, \emph{On the security of
        the permuted kernel identification scheme}, Advances in Cryptology --
    {CRYPTO} '92 (Santa Barbara, CA, USA) (E.~F. Brickell, ed.), Lecture Notes in
    Computer Science, vol. 740, Springer, Heidelberg, Germany, August~16--20,
    1993, pp.~305--311.

    \bibitem{MR1471991}
    C.~H. Bennett, E.~Bernstein, G.~Brassard, and U.~Vazirani, \emph{Strengths and
        weaknesses of quantum computing}, {SIAM} J. Comput. \textbf{26} (1997),
    no.~5, 1510--1523.

    \bibitem{Bettaieb2024}
    S.~Bettaieb, L.~Bidoux, V.~Dyseryn, A.~Esser, P.~Gaborit, M.~Kulkarni, and
    M.~Palumbi, \emph{{PERK:} compact signature scheme based on a new variant of
        the permuted kernel problem}, Designs, Codes and Cryptography (2024), online
    first.

    \bibitem{EC:Beullens20}
    W.~Beullens, \emph{Sigma protocols for {MQ}, {PKP} and {SIS}, and {Fishy}
        signature schemes}, Advances in Cryptology -- {EUROCRYPT}~2020, Part~III
    (Zagreb, Croatia) (A.~Canteaut and Y.~Ishai, eds.), Lecture Notes in Computer
    Science, vol. 12107, Springer, Heidelberg, Germany, May~10--14, 2020,
    pp.~183--211.

    \bibitem{INDOCRYPT:BFKMPP19}
    W.~Beullens, J.-C. Faug{\`e}re, E.~Koussa, G.~{Macario-Rat}, J.~Patarin, and
    L.~Perret, \emph{{PKP}-based signature scheme}, Progress in Cryptology -
    INDOCRYPT~2019: 20th International Conference in Cryptology in India
    (Hyderabad, India) (F.~Hao, S.~Ruj, and S.~{Sen Gupta}, eds.), Lecture Notes
    in Computer Science, vol. 11898, Springer, Heidelberg, Germany,
    December~15--18, 2019, pp.~3--22.

    \bibitem{10.1007/978-3-031-33017-9_2}
    L.~Bidoux and P.~Gaborit, \emph{Compact post-quantum signatures from proofs
        of knowledge leveraging structure for the {PKP}, {SD} and {RSD}
        problems}, Codes, Cryptology and Information Security (Cham) (S.~El~Hajji,
    S.~Mesnager, and E.~M. Souidi, eds.), Springer Nature Switzerland, 2023,
    pp.~10--42.

    \bibitem{cryptoeprint:2023/589}
    A.~Esser, J.~Verbel, F.~Zweydinger, and E.~Bellini,
    \emph{$\texttt{CryptographicEstimators}$: a software library for
        cryptographic hardness estimation}, Cryptology ePrint Archive, Paper
    2023/589, 2023, \url{https://eprint.iacr.org/2023/589}.

    \bibitem{MR1533848}
    S.~D. Fisher and M.~N. Alexander, \emph{Classroom notes: {M}atrices over a
        finite field}, Amer. Math. Monthly \textbf{73} (1966), no.~6, 639--641.

    \bibitem{MR2483235}
    P.~Flajolet and R.~Sedgewick, \emph{Analytic combinatorics}, Cambridge
    University Press, Cambridge, 2009.

    \bibitem{JC:Georgiades92}
    J.~Georgiades, \emph{Some remarks on the security of the identification scheme
        based on permuted kernels}, Journal of Cryptology \textbf{5} (1992), no.~2,
    133--137.

    \bibitem{MR1397498}
    R.~L. Graham, D.~E. Knuth, and O.~Patashnik, \emph{Concrete mathematics},
    second ed., Addison-Wesley Publishing Company, Reading, MA, 1994, A
    foundation for computer science.

    \bibitem{PKC:JauJou01}
    {\'E}.~Jaulmes and A.~Joux, \emph{Cryptanalysis of {PKP}: A new approach},
    PKC~2001: 4th International Workshop on Theory and Practice in Public Key
    Cryptography (Cheju Island, South Korea) (Kwangjo Kim, ed.), Lecture Notes in
    Computer Science, vol. 1992, Springer, Heidelberg, Germany, February~13--15,
    2001, pp.~165--172.

    \bibitem{cryptoeprint:2019/412}
    E.~Koussa, G.~Macario-Rat, and J.~Patarin, \emph{On the complexity of the
        permuted kernel problem}, Cryptology ePrint Archive, Paper 2019/412, 2019,
    \url{https://eprint.iacr.org/2019/412}.

    \bibitem{EPRINT:LAMPAT11}
    R.~Lampe and J.~Patarin, \emph{Analysis of some natural variants of the {PKP}
        algorithm}, Cryptology ePrint Archive, Report 2011/686, 2011,
    \url{https://eprint.iacr.org/2011/686}.

    \bibitem{NIST_call}
    {NIST}, \emph{{NIST} announces additional digital signature candidates for the
        {PQC} standardization process}, online, July 2023,
    \url{https://csrc.nist.gov/news/2023/additional-pqc-digital-signature-candidates}.

    \bibitem{ACNS:PaiTer21}
    T.~B. Paiva and R.~Terada, \emph{Cryptanalysis of the binary permuted kernel
        problem}, ACNS 21: 19th International Conference on Applied Cryptography and
    Network Security, Part~II (Kamakura, Japan) (K.~Sako and N.~O. Tippenhauer,
    eds.), Lecture Notes in Computer Science, vol. 12727, Springer, Heidelberg,
    Germany, June~21--24, 2021, pp.~396--423.

    \bibitem{C:ParCha93}
    J.~Patarin and P.~Chauvaud, \emph{Improved algorithms for the permuted kernel
        problem}, Advances in Cryptology -- {CRYPTO} '93 (Santa Barbara, CA, USA)
    (D.~R. Stinson, ed.), Lecture Notes in Computer Science, vol. 773, Springer,
    Heidelberg, Germany, August~22--26, 1994, pp.~391--402.

    \bibitem{PERK_website}
    {PERK Team}, \emph{{PERK} {O}fficial {W}ebsite}, online, May 2023,
    \url{https://pqc-perk.org}.

    \bibitem{https://doi.org/10.1002/ett.4460080505}
    G.~Poupard, \emph{A realistic security analysis of identification schemes based
        on combinatorial problems}, European Transactions on Telecommunications
    \textbf{8} (1997), no.~5, 471--480.

    \bibitem{MR1824028}
    B.~E. Sagan, \emph{The symmetric group}, second ed., Graduate Texts in
    Mathematics, vol. 203, Springer-Verlag, New York, 2001, Representations,
    combinatorial algorithms, and symmetric functions.

    \bibitem{9834867}
    P.~Santini, M.~Baldi, and F.~Chiaraluce, \emph{A novel attack to the permuted
        kernel problem}, 2022 IEEE International Symposium on Information Theory
    (ISIT), 2022, pp.~1441--1446.

    \bibitem{10.1109/TIT.2023.3323068}
    P.~Santini, M.~Baldi, and F.~Chiaraluce, \emph{Computational hardness of the
        permuted kernel and subcode equivalence problems}, IEEE Transactions on
    Information Theory \textbf{70} (2023), no.~3, 2254--2270.

    \bibitem{C:Shamir89}
    A.~Shamir, \emph{An efficient identification scheme based on permuted kernels},
    Advances in Cryptology -- {CRYPTO} '89 (Santa Barbara, CA, USA) (G.~Brassard,
    ed.), Lecture Notes in Computer Science, vol. 435, Springer, Heidelberg,
    Germany, August~20--24, 1990, pp.~606--609.

\end{thebibliography}
\end{document}